\documentclass[oneside,english]{amsart}
\usepackage[T1]{fontenc}
\usepackage[latin9]{inputenc}
\usepackage{amsthm}
\usepackage{amstext}
\usepackage{amssymb}

\makeatletter
\numberwithin{equation}{section}
\numberwithin{figure}{section}
\theoremstyle{plain}
\newtheorem{thm}{\protect\theoremname}
  \theoremstyle{plain}
  \newtheorem{prop}[thm]{\protect\propositionname}
  \theoremstyle{remark}
  \newtheorem{rem}[thm]{\protect\remarkname}
  \theoremstyle{definition}
  \newtheorem{defn}[thm]{\protect\definitionname}
  \theoremstyle{plain}
  \newtheorem{lem}[thm]{\protect\lemmaname}

\usepackage{babel}
  \providecommand{\definitionname}{Definition}
  \providecommand{\lemmaname}{Lemma}
  \providecommand{\propositionname}{Proposition}
  \providecommand{\remarkname}{Remark}
\providecommand{\theoremname}{Theorem}

\author{C\'esar Camacho}
\email{camacho@impa.br}
\address{IMPA, Estrada Dona Castorina 110, Rio de Janeiro, Brasil.}
\author{Rudy Rosas}

\email{rudy.rosas@pucp.pe}

\address{Pontificia Universidad Cat\'olica del Per\'u, Av Universitaria 1801, Lima, Per\'u.}
\address{Instituto de Matem\'atica y Ciencias Afines, Jr. los bi\'ologos 245,  Lima, Per\'u}

\begin{document}

\title{Invariant sets near singularities of holomorphic foliations}

\maketitle

\begin{abstract} Consider a complex one dimensional foliation on a complex surface near a singularity $p$. 
If $\mathcal I$ is a closed invariant set containing  the singularity $p$, then $\mathcal I$  contains either a 
separatrix at $p$ or an invariant real three dimensional manifold singular at $p$.
\end{abstract}

We consider a one-dimensional holomorphic foliation $\mathcal{F}$ on a complex surface $V$, with a singularity 
at $p\in V$. We assume that  $p\in V$ is a normal singularity whose resolution graph is a tree; in particular we admit 
the possibility that $p\in V$ be a regular point of $V$.  It is proved in \cite{Camacho} and \cite{CS} that for any such 
singularity there exists at least one separatrix of the foliation, that is, there is a germ at $p\in V$ of an analytic curve 
invariant by the foliation.
In this paper we study the nature of the closed $\mathcal{F}$-invariant sets near the singularity and their relations with 
the separatrices of $\mathcal{F}$.
When $(p,V)=(0,\mathbb{C}^{2})$ the foliation $\mathcal F$ can be represented by a localholomorphic vector field and 
the singularity is called {\it reduced} if the linear part of the vector field at $\text{\ensuremath{p}}$ has eigenvalues 
$\lambda_{1},\lambda_{2}\in\mathbb{C}$ with $\lambda_{1}\neq0$ and such that 
$\lambda=\frac{\lambda_{2}}{\lambda_{1}}$ is not a rational positive number. 
This last number will be called the eigenvalue of the singularity. The singularity $p$ is hyperbolic if
$\lambda\in\mathbb{C}\backslash\mathbb{R}$,
it is a saddle if $\lambda<0$, it is a node if $\lambda>0$, and it
is a saddle-node if $\lambda=0$. When the singularity of $\mathcal{F}$
at $0\in\mathbb{C}^{2}$ is a node we have a particular kind of invariant
sets. In this case there are suitable local coordinates such that the foliation
near $p\in V$ is given by the holomorphic vector field
$x\frac{\partial}{\partial x}+\lambda y\frac{\partial}{\partial y}$
and we have the multi-valued first integral $yx^{-\lambda}$. Then
the closure of any leaf other than the separatrices is a set of
type $|y|=c|x|^{\lambda}$, $c>0$, which is called a {\it nodal separator}
\cite{MM}. More precisely, we say that a set $S$ is a nodal separator for a node, if in
linearizing coordinates as above we have $S=\{(x,y):|y|=c|x|^{\lambda}\}\cap B$, $c>0$, where in those coordinates $B$ is 
an open ball centered at the singularity. Clearly $S$ is  invariant by the foliation restricted to $B$. In general, we say that a 
subset $p\in S\subset V$ is a {\it nodal separator at p}
if the strict transform of $S$ in the resolution%
\footnote{By a resolution we always means the minimal resolution.%
} of $\mathcal{F}$ is a nodal separator for some node in the resolution.
We recall that a germ $(S,p)$ of analytic irreducible curve at $p$ is called a separatrix at $p$ if it is invariant by $\mathcal{F}$. 
We say that $S_1$ is a {\it representative of the separatrix} $(S,p)$ if 
$S_1=\phi(\mathbb{D})$\footnote{$\mathbb{D}=\{z\in\mathbb{C}:|z|<1\}$} 
for a holomorphic injective map $\phi:\mathbb{D}\rightarrow V$ regular on 
$\mathbb{D}\backslash\{0\}$ and such that the germ of $S_1$ at $p$ coincides with $(S,p)$.

\begin{thm}
\label{principal}
Let $\mathcal I\subset V$  be a closed $ \mathcal{F}$-invariant subset such that $p\in \mathcal I\ne \{p\}$.  Then $\mathcal I$ contains a separatrix or a nodal separator at $p$.
\end{thm}
In particular, if $\mathcal{F}$ does not contain nodes in its resolution
then $\mathcal I$ contains a separatrix. In fact, in this case
we obtain a slightly stronger result. A {\it dicritical component} of the
resolution of $\mathcal{F}$ is an irreducible component of the
divisor everywhere transverse to the reduced lifted foliation.
We say that the foliation $\mathcal F$
is { \it non-dicritical} whenever all the components of the resolution divisor are
not dicritical.
\begin{thm}
\label{principal2} Suppose that $\mathcal{F}$ is non-dicritical and contains
no nodes in its resolution at $p$. Let  $S_{1},...,S_{r}\subset V$, $r\ge1$, be representatives of the
separatrices of $\mathcal{F}$ at $p$ and for each $j=1,...,r$  take a complex
disc $\Sigma_{j}\subset V$ passing transversely through a  point in
$S_{j}\backslash\{p\}$.  Then the set
\[
Sat_{\mathcal{F}}(\bigcup_{j=1}^{r}\Sigma_{j})
\]
is a punctured neighborhood of the singularity $p\in V$.
\end{thm}
This is a generalization of the same statement  in \cite{MM2} for a generalized curve
$\mathcal{F}$ (as defined in \cite{CSL}). In Theorem \ref{principaDetallado}
we give a more detailed version of this theorem.
\begin{rem}Theorems \ref{principal} and \ref{principal2} are of local nature. If  $U$ is a neighborhood 
of $p$ in $V$, we can apply these theorems to $U$ and $\mathcal{F}|_U$ instead of $V$ and 
$\mathcal{F}$. Thus, for example in Theorem \ref{principal2}, in order to obtain a punctured neighborhood 
of $p$ it suffices to saturate the sets $\Sigma_j$ by the leaves of $\mathcal{F}|_U$. Of course in this case 
we must take the $\Sigma_j$ in $U$.
\end{rem}
One of the ingredients in the proof of these results is Theorem \ref{separatriz},
which we use to relate the distribution of the separatrices and the
saddle-node corners in the resolution of the foliation. A particular
version of this theorem is proved in \cite{Camacho} and it is essentially
the proof of the existence of separatrices. As an application of Theorem \ref{separatriz} we give another proof
of the so called Strong Separatrix Theorem proved in \cite{ORV}.

\section{weighted graph of a resolution and the existence of separatrices}

A \emph{weighted graph }is a connected graph $\Gamma$ with vertices
$\{v_{1},...,v_{m}\}$ and at each $v_{j}$ a nonzero number $w(v_{j})$,
called the weight of $v_{j}$. The intersection matrix $(\alpha_{ij})$
of $\Gamma$ is defined by $\alpha_{jj}=w(v_{j})$ and $\alpha_{ij}=$number
of edges joining $v_{i}$ and $v_{j}$ if $i\neq j$.

Let M be a complex regular surface and $D=\bigcup_{j\in J}P_{j}\subset M$
a finite union of compact regular Riemann surfaces with normal crossings.
Each intersection $P_i\cap P_j\ne\emptyset$, $i\ne j$, is called a {\it corner}.
The weighted graph $\Gamma$ associated to $D$ is composed by the
vertices $\{P_{j}\}_{j\in J}$, with $w(P_{j})=P_{j}.P_{j}$ and such that there is
an edge between $P_{i}$ and $P_{j}$ for each point in the intersection
$P_{i}\cap P_{j}$.
In this case the intersection matrix is given by $\alpha_{jj}=P_{j}.P_{j}$
and $\alpha_{ij}=\mbox{card}(P_{i}\cap P_{j})$ if $i\neq j$.

Let $\mathcal {\tilde F}$ be a holomorphic foliation with reduced singularities
defined on a neighborhood of D in M and such that D is invariant by
$\mathcal {\tilde F}$.

\begin{thm}
\label{separatriz}Suppose that $\Gamma$ is a tree with negative
definite intersection matrix. Then there exists a singular point
$q\in P_{j_{0}},j_{0}\in J$,
outside the set of corners,
such that $\emph{Re {CS}}(P_{j_{0}},q)<0$.
\end{thm}

The complex number $\mbox{CS}(P_{j_{0}},q)$ is the Camacho-Sad's index of $q$ relative to $P_{j_{0}}$.
This theorem applies to the case where $\mathcal{F}$
is a holomorphic foliation defined on $V$ with a singularity at $p\in V$. In this
case $\pi:M\rightarrow V$ is the resolution of $\mathcal{F}$ at $p\in V$. Then
$M$ is a regular complex surface and $\pi^{-1}(p)=\bigcup_{j\in J}P_{j}$ is a
finite union of compact Riemann surfaces $P_{j}$ with normal crossings and
$\mathcal{\tilde F}$ is the pull back foliation $\pi^{*}{\mathcal{F}}$. It is well known
(\cite{Du Val} , \cite{Laufer}, \cite{Mumford}) that the intersection matrix of the
associated divisor is negative definite.
Thus the above theorem generalizes the following theorem proved in \cite{Camacho},
\begin{thm}
\label{camacho}Suppose $V$ is a complex normal irreducible surface
such that its resolution graph at $p$ is a tree. If $\pi^{-1}(p)=\bigcup_{j\in J}P_{j}$
contains no dicritical components, then there is a singular point
$q\in P_{j_{0}},j_{0}\in J$, not a corner, such that
$\emph{CS}(P_{j_{0}},q)\notin\mathbb{R}_{\geq0}$.
\end{thm}
This implies the existence of
separatrices for a holomorphic foliation defined on $V$, singular
at $p\in V$, whenever $V$ is a normal irreducible surface whose
resolution graph at $p$ is a tree.
As in \cite{Camacho}, the key for the proof of Theorem 4 is the following proposition
(\cite{Camacho}, Proposition 2.1).
First we fix any vertex $m$ of $\Gamma$ and introduce an ordering
of $\Gamma$ such that:
\begin{enumerate}
\item At any vertex there is at most one edge getting out
\item $m$ is a maximal element i.e. no edge gets out of $m$.
\end{enumerate}
Clearly $m$ is the sole maximal element. Given any vertex $v$ we
define $n(v)$ as the number of positive edges between $v$ and $m$.
The size of $\Gamma$ is $s=\mbox{max}_{v}n(v)$ and $v$ is at level
$l$ if $l=s-n(v)$.
\begin{prop}
\label{arbol}Let $\Gamma$ be an ordered weighted tree with negative
definite intersection matrix. Then there is a real negative map $h$
defined on the set of vertices of $\Gamma$ such that
\begin{enumerate}
\item If $v$ is a minimal element of $\Gamma$, then $h(v)=w(v)$.
\item If $v$ is not minimal and $\{v_{i_{1}},...,v_{i_{r}}\}$ is the set
of its immediate predecessors, then
\[
h(v)=w(v)-\sum_{j=1}^{r}\frac{1}{h(v_{i_{j}})}.
\]

\end{enumerate}
\end{prop}

\subsection{Proof of Theorem \ref{separatriz}}

Proceeding by contradiction assume that any singular point $z$, not
a corner, is such that $\mbox{Re}(\mbox{CS}(D,z))\ge0$. Take any
component $P$ of $D$ and consider the order of $\Gamma$ such that
$P$ is a maximal element. Suppose $Q_{r+1}$ is a vertex at level
one, $Q_{1},...,Q_{r}$ all its predecessors (at level zero) and $Q_{r+2}$
the only successor of $Q_{r+1}$ at level two. Write $q_{j}=Q_{r+1}\cap Q_{j}$,
$j=1,...,r+2$, $j\neq r+1$. By the Camacho-Sad's formula and Proposition
\ref{arbol}
\[
\mbox{ReCS}(Q_{j},q_{j})\le w(Q_{j})=h(Q_{j})<0\mbox{ for }j=1,...,r.
\]

Since the singular points are reduced,

\[
\mbox{ReCS}(Q_{r+1},q_{j})=\mbox{Re}(\frac{1}{\mbox{CS}(Q_{j},q_{j})})\mbox{ or }\mbox{ReCS}(Q_{r+1},q_{j})=0>\mbox{Re}(\frac{1}{\mbox{CS}(Q_{j},q_{j})}).
\]

In any case
\[
\mbox{Re}\mbox{CS}(Q_{r+1},q_{j})\ge\mbox{Re}(\frac{1}{\mbox{CS}(Q_{j},q_{j})})\ge\frac{1}{\mbox{ReCS}(Q_{j},q_{j})}\ge\frac{1}{h(Q_{j})}.
\]

Again by Camacho-Sad's formula
\[
\mbox{ReCS}(Q_{r+1},q_{r+2})+\sum_{j=1}^{r}\mbox{ReCS}(Q_{r+1},q_{j})\le w(Q_{r+1}),
\]

\[
\mbox{ReCS}(Q_{r+1},q_{r+2})\le w(Q_{r+1})-\sum_{j=1}^{r}\frac{1}{h(Q_{j})}=h(Q_{r+1})<0,
\]
so
\[
\mbox{ReCS}(Q_{r+2},q_{r+2})\ge\frac{1}{\mbox{ReCS}(Q_{r+1},q_{r+2})}\ge\frac{1}{h(Q_{r+1})}.
\]
Proceeding by induction, if $P_{1},...,P_{l}$ are all the immediate
predecessors of $P$ and $p_{j}=P\cap P_{j}$ for $j=1,...,l$, then
\[
\sum_{j=1}^{l}\mbox{ReCS}(P,p_{j})\ge\sum_{j=1}^{l}\frac{1}{h(P_{j})}=w(P)-h(P)>w(P)
\]
which is a contradiction.

\subsection{Strong Separatrix Theorem}

If $V$ is as in Theorem \ref{camacho}, we can apply Theorem
$\ref{separatriz}$ to any connected
non-dicritical subgraph $\Gamma'$ of the resolution graph $\Gamma$
of $\mathcal{F}$. In this case, if $\{P_{j}\}_{j\in J'}$, $J'\subset J$
are the vertices of $\Gamma'$, we find a singular point $q\in P_{j_{0}},j_{0}\in J',q\notin P_{j},j\in J'\backslash\{j_{0}\}$
such that $\mbox{ReCS}(P_{j_{0}},q)<0$. Of course it might happen
that $q\in P_{j}$ with $j\in J\backslash J'$. Observe that the foliation
may be dicritical, that is, some components $P_{j}$, $j\in J\backslash J'$
may be non-invariant. As an application of these considerations we
give another proof of the following theorem proved in \cite{ORV}.
\begin{thm}
\textbf{\emph{(Strong Separatrix Theorem)}}\label{strongCS}
The number of separatrices of an isolated singularity of a planar
holomorphic vector field is greater than the number of nodal corner
points in its resolution.
\end{thm}
We say that a reduced singular point $z$ is \emph{non-negative}%
\footnote{\emph{This definition includes the bad points defined in \cite{ORV}.}%
} if $z$ has two separatrices $S_{1},S_{2}$ with $\mbox{Re}(\mbox{CS}(S_{i},z))\geq0$,
$i=1,2$. If $z$ is a saddle-node the above condition means that
the central manifold of $z$ has Camacho-Sad index with nonnegative
real part. If $z$ is not a saddle-node, that is, if $z$ has eigenvalue
$\lambda\neq0$, then $z$ is non-negative if $\mbox{Re}(\lambda)\geq0$\footnote{Observe that in this case we have $\mbox{sgnReCS}(\lambda)=\mbox{sgnReCS}(\frac{1}{\lambda})$. Then $\mbox{sgnReCS}(\lambda)$ is well defined and is an analytical invariant.} and it is either a node or a hyperbolic singularity. Now, it is sufficient to prove the following
theorem, which improves Theorem 1.2 of \cite{ORV}.
\begin{thm}
\label{CSref} Let $D$ be the exceptional divisor of the resolution at
$p\in V$ of $\mathcal{F}$.
Then there exist at least one separatrix issuing from each connected
component of
\[
D_{*}=D-\{\mbox{non-negative singularities}\}-\{\mbox{dicritical components}\}
\]
\end{thm}
\begin{proof}
Let $D=\bigcup_{j\in J}P_{j}$ and let $D'=\bigcup_{j\in J'}P_{j'}$,
$J'\subset J$ be the closure of a connected component of $D_{*}$.
If we apply Theorem \ref{separatriz} to the graph associated to $D'$
we obtain a singular point $q\in P_{j_{0}},j_{0}\in J',q\notin P_{j},j\in J'\backslash\{j_{0}\}$
such that $\mbox{ReCS}(P_{j_{0}},q)<0$. This last inequality shows
that $q\notin P_{j}$, $j\in J\backslash J'$, so $q$ gives a separatrix. \end{proof}
\begin{rem}
\label{strong separatrix}Observe that Theorem \ref{CSref} gives a
especial kind of separatrices. In particular, if $S$ is one such
separatrix, then:
\begin{enumerate}
\item $S$ is the strong manifold of a saddle node in the resolution, or
\item $S$ passes through a singular point in the resolution with eigenvalue
$\lambda$ such that $\mbox{ReCS}(\lambda)<0$.
\end{enumerate}
The separatrices described in items (1) and (2) above will be called \emph{strong
separatrices}. Any other separatrix will be called \emph{weak separatrix}
\footnote{Strong and weak separatrices are defined only if $\mathcal{F}$ is non-reduced.%
}.
\end{rem}
\begin{rem} Since each nodal singularity, not in a corner, yields a separatrix we can replace in Theorem \ref{strongCS} the words "nodal corner" by "nodal". In fact, we can replace "nodal corner" by "non-positive".
\end{rem}

\section{Approximation chains}

Let $D=\bigcup_{j\in J}P_{j}$ be the exceptional divisor of the resolution of
$\mathcal{F}$ at $p\in V$.
\begin{defn}
An {\it approximation chain} for $\mathcal{F}$ is a sequence of invariant
components $P_{j_{1}},...,P_{j_{n}}$ ($j_{1},...,j_{n}\in J,n\in\mathbb{N}$)
with the following properties:
\begin{enumerate}
\item For each $k=1,...,n-1$ we have that $P_{j_{k}}$ intersects $P_{j_{k+1}}$at
a point $z_{k}$,
\item The singularity at $z_{k}$ is not a node,
\item If the singularity at $z_{k}$ is a saddle-node then its strong manifold
is contained in $P_{j_{k+1}}$,
\item $P_{j_{n}}$ contains a singularity $q$, not a corner, such that
$\mbox{Re CS}(P_{j_{n}} , q) < 0$.

\end{enumerate}
We call $P_{j_{1}}$ the {\it starting component} of the approximation chain.
\end{defn}

The above definition is justified by the following dynamical property.
\begin{prop}
\label{dynamical}Let $q$ be as in (4) in the definition above and let $S$
be the separatrix through $q$ transverse to $P_{j_{n}}$. Let $\{x_{n}\}_{n\in\mathbb{N}}$
be a sequence of points outside $D$ such that $x_{n}$ tends to a
regular point in $P_{j_{1}}$. Then for each $n\in\mathbb{N}$ there
is a point $y_{n}$ in the leaf through $x_{n}$ such that $y_{n}$
tends to a regular point in $S$.
\end{prop}
The proof of this proposition is given below.
The main result of this section is the following existence theorem
for approximation chains.
\begin{thm}
\label{existencia de cadenas}Any non-dicritical component of the
resolution of $\mathcal{F}$ is the starting component of an approximation
chain. \end{thm}
\begin{proof}
We introduce an order in the resolution graph $\Gamma$ of $\mathcal{F}$
as follows: If $P_{i}\cap P_{j}\neq\emptyset$ and $z=P_{i}\cap P_{j}$
is a saddle-node, then $P_{i}<P_{j}$ whenever the strong manifold
of $z$ is contained in $P_{j}$. If $z=P_{i}\cap P_{j}\neq\emptyset$
and $z$ has eigenvalue $\lambda\notin[0,+\infty)$, then $P_{i}=P_{j}$.
Clearly the set of maximal elements is not empty. Let $\mathcal{M}$
be the union of all maximal elements and $\mathcal{M}_{c}$ a connected
component of $\mathcal{M}\backslash\{\mbox{nodal corners}\}$. Then
$\overline{\mathcal{M}_{c}}={\displaystyle \bigcup_{j\in J_{c}}}P_{j}$,
$J_{c}\subset J$ and the components $\{P_{j}\}_{j\in J_{c}}$ are
all of the same order. This means that whenever $i,j\in J_{c}$ and
$P_{i}\cap P_{j}\neq\emptyset$ then this intersection has eigenvalue
$\lambda\notin[0,+\infty)$. Moreover, if $z=P_{i}\cap P_{j}$ and
$i\in J_{c},j\notin J_{c}$ we have the following possibilities:
\begin{enumerate}
\item $z$ is a node,
\item $z$ is a saddle-node with strong manifold contained in $P_{i}$,
or
\item $P_{j}$ is a dicritical component, so $z$ is not singular for the
foliation.
\end{enumerate}
By Theorem \ref{separatriz} there is a singular point $q\in P_{j_{0}}$,
$j_{0}\in J_{c}$, $j_{0}\notin P_{j}$, $j\in J_{c}\backslash\{j_{0}\}$
such that $\mbox{Re CS}(P_{j_{0}},q)<0$. Therefore the possibilities
(1),(2),(3) above show that $q\notin P_{j}$ for $j\notin J_{c}$, so $q$
is not a corner.

Therefore, given any invariant component $P$ in $D$, we find a sequence
$P=P_{j_{1}}\le...\le P_{j_{n}}$ such that $P_{j_{n}}$ contains
a singularity $q$, not a corner, such that $\mbox{Re CS}(P_{j_{n}},q)<0$.
Clearly $P_{j_{1}},...,P_{j_{n}}$ is an approximation chain.
\end{proof}

For the proof of Proposition  \ref{dynamical} we need the following Lemmas.
\begin{lem}
\label{silla}Let $\mathcal{F}$ be a holomorphic foliation with a
reduced singularity at
$0\in\mathbb{C}^{2}$ and eigenvalue $\lambda\notin[0,\infty)$. Let $S$ be a separatrix through $0\in\mathbb{C}^{2}$.
Let $(p_{n})_{n\in\mathbb{N}}$ be a sequence of points outside the
separatrices such that $p_{n}\rightarrow0$ as $n\rightarrow\infty$.
Then, for each $n\in\mathbb{N}$ there exists a point $q_{n}$ in
the leaf through $p_{n}$ such that $q_{n}$ tends to a  point
in $S\backslash\{0\}$ as $n\rightarrow\infty$. \end{lem}
\begin{proof}
In a neighborhood of $0\in\mathbb{C}^{2}$ the foliation $\mathcal{F}$
is generated by a holomorphic vector field $Z$
\[
Z=\lambda_{1}x(1+\cdots)\frac{\partial}{\partial x}+\lambda_{2}y(1+\cdots)\frac{\partial}{\partial y},
\]
 with $Re(\lambda_{1})>0>Re(\lambda_{2})$ and such that $S$ is given
by $\{y=0\}$. Then in a neighborhood $U$ of $0\in\mathbb{C}^{2}$
we have $Z=xA\frac{\partial}{\partial x}+yB\frac{\partial}{\partial y}$
with $Re(A)>0>Re(B)$. Let $\phi$ be the real flow associated to
$Z$ and let $a,b>0$ be such that
\[
\{|x|\leq a,|y|\leq b\}\subset U.
\]
 Put $\phi(t,p_{n})=(x(t),y(t))$ and $g(t)=|x(t)|^{2}$. A straightforward
computation shows that
\[
g'(t)=2|x(t)|^{2}Re\{A(t)\}>0,
\]
 hence the function $|x(t)|$ is strictly increasing. Analogously
we prove that $|y(t)|$ is strictly decreasing. Thus, since $p_{n}=(x_{0},y_{0})$
with $|x_{0}|\leq a$ and $|y_{0}|\leq b$ we have that the orbit
of $p_{n}$ intersects the set $\{|x|=a,|y|\le b\}$ at exactly one
point $q_{n}$. Finally it is easy to prove that $q_{n}$ tends to
a point in $\{|x|=a,y=0\}\subset S$.
\end{proof}

\begin{lem}
\label{silla nodo} Let $\mathcal{F}$ be a holomorphic foliation
with a saddle-node singularity at $0\in\mathbb{C}^{2}$. There is
a neighborhood $U$ of $0\in\mathbb{C}^{2}$ with the following property:
If $L$ is a leaf of $\mathcal{F}|_{U}$ other than the central manifold,
then $\overline{L}$ contains the strong manifold of the saddle-node.\end{lem}
\begin{proof} This an easy consequence of the sectorial normalization theorem of Hukuhara-Kimura-Matuda (\cite{HKM}).
We can find a simply statement of this theorem in \cite{Loray}.
\end{proof}

Clearly Proposition \ref{dynamical} follows by successive applications
of the following lemma, which is a direct consequence of Lemmas \ref{silla}
and \ref{silla nodo}.
\begin{lem}
Let $\mathcal{F}$ be a holomorphic foliation with a reduced singularity
at $0\in\mathbb{C}^{2}$ having two separatrices $S_{1}$ and $S_{2}$
 and such that one of the following conditions holds:
\begin{enumerate}
\item The singularity at $0\in\mathbb{C}^{2}$ has eigenvalue $\lambda\notin[0,\infty)$,
\item The singularity at $0\in\mathbb{C}^{2}$ is a saddle-node with $S_{2}$
as its strong manifold.
\end{enumerate}
Let $(x_{n})_{n\in\mathbb{N}}$ be a sequence of points outside the separatrices $S_1$, $S_2$ and
such that $x_{n}$ tends to a point in $S_{1}\backslash\{0\}$ as $n\rightarrow\infty$.
Then, for each $n\in\mathbb{N}$ there exists a point $y_{n}$ in
the leaf through $x_{n}$ such that $y_{n}$ tends to a point
in $S_{2}\backslash\{0\}$. \end{lem}

\section{Proof of Theorems \ref{principal} and \ref{principal2}}

\subsection{Proof of Theorem \ref{principal2}} \label{section detallado}

Clearly Theorem \ref{principal2} is a corollary of the following theorem.
\begin{thm}
\label{principaDetallado}Let $\pi:M\rightarrow V$ be the resolution of
$\mathcal{F}$ at $p\in V$
and let $D=\pi^{-1}(0)$. Consider a connected component $D_{c}$
of
\[
D-\{\mbox{nodal singularities}\}-\{\mbox{dicritical components}\}
\]
 and set
\[
\widetilde{D_{c}}=\overline{D_{c}}-\{\mbox{nodal singularities}\}
\]
 Let $S_{1},...,S_{r}$, $r\ge1$ be representatives   of the strong separatrices%
\footnote{See remark \ref{strong separatrix}}.%
 issuing from $D_{c}$ and for each $j=1,...,r$ take a complex disc
$\Sigma_{j}$ passing transversely through a  point of $S_{j}\backslash\{p\}$.
Let $\mathcal{S}$ be a union of representatives of the weak separatrices issuing
from $D_{c}$. Then the set
\[
Sat_{\mathcal{F}}(\bigcup_{j=1}^{r}\Sigma_{j})\cup\mathcal{S}\cup\widetilde{D_{c}}
\]
is a neighborhood of $\widetilde{D_{c}}$ in $M$. In particular,
if $\mathcal{F}$ is a non-dicritical foliation without nodes in its
resolution, then
\[
Sat_{\mathcal{F}}(\bigcup_{j=1}^{r}\Sigma_{j})\cup\mathcal{S}
\]
 gives a punctured neighborhood of the singularity $p$.\end{thm}
\begin{proof}
Let $B_{1},...,B_{l}$ be small open balls centered at the nodal singularities
in $D$ and set
\[
K=\widetilde{D_{c}}-(B_{1}\cup...\cup B_{l})
\]
\[
\Omega=Sat_{\mathcal{F}}(\bigcup_{j=1}^{r}\Sigma_{j})\cup\mathcal{S}\cup\widetilde{D_{c}}.
\]
Clearly $K$ is compact. It is easy to see that, if the balls $B_{i}$
are small enough and $\Omega$ contains a neighborhood of $K$, then
$\Omega$ contains a neighborhood of $\widetilde{D_{c}}$. Thus, suppose
by contradiction that $\Omega$ contains no neighborhood of $K$.
Then we find a sequence $(x_{n})_{n\in\mathbb{N}}$ with $x_{n}\rightarrow\zeta\in K$
and such that $x_{n}\notin\Omega$ for all $n\in\mathbb{N}$. Then
\begin{enumerate}
\item $\zeta$ is a regular point contained in some non-dicritical component,
or
\item $\zeta$ is a non-nodal singularity in $D_{c}$.
\end{enumerate}
In view of Theorem \ref{existencia de cadenas} and Proposition \ref{dynamical},
in case (1) above we deduce that for $n$ big enough the leaf through
$x_{n}$ is arbitrarily close to a strong separatrix, which is contradiction.
Suppose that $\zeta$ is a singularity with eigenvalue $\lambda\notin[0,\infty)$
or a saddle-node with strong manifold contained in $D$. Since $x_n\notin\Omega$ the points
$x_n$ are outside the separatrices $\{S_j:j=1,...,r\}\cup\mathcal{S}$ issuing from $D_c$. Then we use Lemmas
\ref{silla} and \ref{silla nodo} to obtain a sequence $(y_{n})_{n\in\mathbb{N}}$,
with $y_{n}$ contained in the leaf through $x_{n}$ and such that
$y_{n}$ tends to a regular point in $D$, therefore we are again
in case (1). Finally, suppose $\zeta$ is a saddle-node with central
manifold contained in $D$. Then the strong manifold of $\zeta$ is
a strong separatrix and Lemma \ref{silla nodo} gives a contradiction.
\end{proof}

\subsection{Proof of Theorem \ref{principal}}

In fact, we will show the following alternative for $\mathcal I$:
\begin{enumerate}

\item[(i)] either there is a neighborhood $U$ of $p\in V$ such that
the set $\mathcal I\cap U$ is a union of a collection of nodal separators
with a collection of representatives of weak separatrices, or
\item[(ii)] $\mathcal I$ contains a strong separatrix.
\end{enumerate}

  Let $\mathcal{N}$ and $\mathcal{D}$ be respectively the set of nodal singularities in $D$ and the set of dicritical components of $D$. Take open sets $B_j$, $j\in\mathcal{N}$ and $T_i$, $i\in\mathcal{D}$, with the following properties:
 \begin{enumerate}
 \item in linearizing coordinates for the node $j\in\mathcal{N}$, we have that $B_j$ is a ball centered at $j$
 \item each $T_i$ is a tubular neighborhood of the dicritical component $i\in\mathcal{D}$
 \item each restriction $\mathcal{F}|_{T_i}$ is a fibration by discs
 \item all the sets $B_j$ and $T_i$ are pairwise disjoint.
 \end{enumerate}
 Set
\[
D^{*}=D-\mathcal{N}-\bigcup_{i\in\mathcal{D}}i,
\]
\[
\widetilde{D}=\overline{D^{*}} -\mathcal{N}.
\]
 Let $\mathcal{S}$ be a union of representatives of the weak separatrices issuing from
$D^{*}$ and let $S_{1},...,S_{r}$ be representatives of the strong separatrices issuing
from $D^{*}$.
Suppose item (ii) of the alternative above does not hold. Then for each $j=1,...,r$ we can take a complex disc $\Sigma_{j}$
disjoint of $\mathcal I$ and passing transversely through a
point of $S_{j}\backslash\{p\}$. Then, if ${\mathcal I}^*$  is the strict transform of $\mathcal I\backslash\{p\}$ in the resolution of $\mathcal{F}$, we have that
\begin{equation}\label{eq1}{\mathcal I}^* \mbox{ is disjoint of }
\Omega:=Sat_{\mathcal{F}}(\bigcup_{j=1}^{r}\Sigma_{j}).
\end{equation}
By applying Theorem \ref{principaDetallado} to each connected component
of $D^{*}$ we obtain that the set
\[
\Omega\cup\mathcal{S}\cup \widetilde{D}
\]
 is a neighborhood of $\widetilde{D}$.

Then \[U:=\Omega\cup\mathcal{S}\cup \widetilde{D}\cup\bigcup_{j\in\mathcal{N}}{B_j}\cup\bigcup_{i\in\mathcal{D}}T_i
\]
is a neighborhood of  the exceptional divisor $D$ and by \ref{eq1} above we deduce that ${\mathcal I}^*$ is contained in
$$\mathcal{S}\cup\bigcup_{j\in\mathcal{N}}{B_j}\cup\bigcup_{i\in\mathcal{D}}(T_i-\Omega).$$
It is easy to see that the intersection  $\overline{{\mathcal I}^*}\cap (T_i-\Omega)$ is a union of representatives of weak separatrices\footnote{Each such one weak separatrix is a fiber of the fibration $\mathcal{F}|_{T_i}$}.
On the other hand, since the intersection ${\mathcal I}^*\cap B_j$ is an invariant set of the node $\mathcal{F}|_{B_j}$, we deduce that
\begin{enumerate}
 \item $\overline{{\mathcal I}^*}\cap B_j$ is a union of nodal separators, or
 \item  $\overline{{\mathcal I}^*}\cap B_j$ is the union of a representative of a weak separatrix with a collection  of nodal separators\footnote{This case can happen if the node $j$ is not a corner}.
 \end{enumerate}
 Therefore $\overline{{\mathcal I}^*}\cap U$ is a union of a collection of nodal separators with  a collection of representatives of weak separatrices, that is, item (i) of the alternative above holds.

 \section{Local minimal sets}

 In this section we give an interpretation of Theorem \ref{principal} in terms of minimal sets. Let $\mathcal{F}$ be a holomorphic foliation
 on a  complex surface $V$ with a singularity at $p\in V$.
 \begin{defn} \label{minimal set}Let $\mathcal{M}$ be a connected closed subset of $V$. We say that $\mathcal{M}$ is a local minimal set at $p$ if there is a neighborhood $U$ of $p$ in $V$ such that the following conditions hold:
 \begin{enumerate}
 \item $p$ is the only singular point of $\mathcal{F}|_U$
 \item  $\{p\}\subsetneq\mathcal{M}$ and  $\mathcal{M}$  is invariant by $\mathcal{F}|_U$
 \item If $\mathcal{M}'\subset\mathcal{M}$ is a connected closed subset of a neighborhood $U'\subset U$ of $p$  with  $\{p\}\subsetneq\mathcal{M}'$ and  $\mathcal{M}'$  invariant by $\mathcal{F}|_{U'}$, then $\mathcal{M}'$ is the connected component of $\mathcal{M}\cap U'$ containing $p$.
 \end{enumerate}
 \end{defn}

\begin{thm}If $p\in V$ is a normal singularity and $\mathcal{M}$ is a local minimal set at $p$, then there exists a neighborhood $\Omega$ of $p$ in $V$ such that
\begin{enumerate}
\item $\mathcal{M}\cap \Omega$ is a representative of a separatrix at $p$, or
\item $\mathcal{M}\cap \Omega$ is a nodal separator at $p$.
\end{enumerate}
\end{thm}
\begin{proof} Let $U$ be the neighborhood of $p$ in $V$ given in Definition \ref{minimal set}. Suppose that $\mathcal{M}$ contains a representative $S$ of a separatrix at $p$. Take a neighborhood $U'\subset U$ of $p$ such that $S$ is closed in $U'$. Then by item (3) of Definition \ref{minimal set} we have that $S$ is the connected component of $\mathcal{M}\cap U'$ containing $p$. Then $(\mathcal{M}\cap U')\backslash S$ is closed in $U�$ and it does not contain the point $p$. Thus there is a neighborhood $\Omega\subset U'$ of $p$ such that $\mathcal{M}\cap\Omega$ is contained in $S$. Clearly we can take $\Omega$ such that $\mathcal{M}\cap\Omega$ is a representative of a separatrix. If $\mathcal{M}$ contains no representative of a separatrix, then by Theorem \ref{principal} we have that  $\mathcal{M}$ contains a nodal separator $S_\mathcal{N}$ at $p$. By repeating the arguments above with $S_\mathcal{N}$ instead of $S$ we obtain a neighborhood $\Omega$ of $p$ such that $\mathcal{M}\cap \Omega$ is a nodal separator at $p$.

\end{proof}

\begin{rem}Let  $\mathcal{M}$ and $\mathcal{M}'$ be local minimal sets at $p$. If there is a neighborhood $\Omega$ of $p$ such that $\mathcal{M}\cap\Omega=\mathcal{M}'\cap\Omega$, then we can think that  $\mathcal{M}$ and $\mathcal{M}'$ are "essentially" the same local minimal set at $p$. In other words, we can define the notion of germ of local minimal set at $p$. In this case the theorem above asserts that the separatrices and the nodal separators are the only germs of local minimal sets at $p$.
\end{rem}


\begin{thebibliography}{References}
\bibitem{Camacho}Camacho, C.: \emph{Quadratic forms and holomorphic
foliations on singular surfaces}, Math. Ann. 282, 177-184 (1988).

\bibitem{CS}Camacho, C., Sad, P.: \emph{Invariant varieties through
singularities of holomorphic vector fields}. Ann. Math. (2) 115(3)
(1982) 579-595.

\bibitem{CSL} Camacho C., Lins A., Sad P.: \emph{Topological
invariants and equidesingularization for holomorphic vector fields},
J. Differential Geometry 20(1984) 143-174.

\bibitem{Du Val} Du Val, P.: \emph{On alsolute and non-absolute singularities
of algebraic surfaces}. Rev. Fac. Sci. Univ. Istanbul, Ser. A 91,
159-215 (1944).

\bibitem{HKM} Hukuhara M., Kimura T.,  Matuda T.: \emph{Equations diff\'erentielles ordinaires
du premier ordre dans le champ complexe}. Publications of the Mathematical
Society of Japan, 7. The Mathematical Society of Japan, Tokyo, 1961.

\bibitem{Laufer} Laufer, H.: \emph{Normal two dimensional singularities}.
Ann. Math. Stud. No. 71 (1971).

\bibitem{Loray} Loray, F.: \emph{Pseudo-groupe d'une singularit\'e
de feuilletage holomorphe en dimension deux.} Pr\'epublication IRMAR,
ccsd-00016434, (2005).

\bibitem{MM} Mar\'in, D., Mattei, J.-F.: \emph{Incompressibilit\'e des
feuilles des germes de feuilletages holomorphes singuliers.} Ann.
Sci. \'Ec. Norm. Sup\'er (4) 41 (2008), 855-903.

\bibitem{MM2} Mar\'in, D., Mattei, J.-F.: \emph{Monodromy and topological classification of germs of holomorphic foliations}, Ann. Sci. \'Ec. Norm. Sup\'er. s\'erie 4, 3 (2012).

\bibitem{Mumford} Mumford, D.: \emph{The topology of normal singularities
of an algebraic surface and a criterion for simplicity}. Publ. Math.
Inst. Hautes Etud. Sci. 9, 5-22 (1961).

\bibitem{ORV} Ortiz-Bobadilla, L., Rosales-Gonz\'ales, E., Voronin,
S. M.: \emph{On Camacho-Sad's Theorem about the existence of a separatrix},
Internat. J. Math. 21 (2010) No. 11, 1413-1420.\end{thebibliography}
\end{document}